\documentclass[12pt,reqno]{amsart}
\usepackage{enumerate, latexsym, amsmath, amsfonts, amssymb, amsthm, graphicx, color}
\def\pmod #1{\ ({\rm{mod}}\ #1)}

\def\F{\Bbb F}

\def\bg{\bigg}
\def\({\bg(}
\def\){\bg)}

\def\Ack{\medskip\noindent {\bf Acknowledgments}}
\theoremstyle{plain}
\newtheorem{theorem}{Theorem}

\newtheorem{lemma}{Lemma}

\newtheorem{conjecture}{Conjecture}
\theoremstyle{definition}

\theoremstyle{remark}
\newtheorem{remark}{Remark}

\makeatletter
\@namedef{subjclassname@2010}{%
	\textup{2010} Mathematics Subject Classification}
\makeatother
\vspace{4mm}
\begin{document}
	\title[additive decompositions of cubes in finite fields]{additive decompositions of cubes in finite fields}
	
	\author[H.-L. Wu and Y.-F. She]{Hai-Liang Wu and Yue-Feng She}
	
	\address {(Hai-Liang Wu) School of Science, Nanjing University of Posts and Telecommunications, Nanjing 210023, People's Republic of China}
	\email{\tt whl.math@smail.nju.edu.cn}
	
	\address {(Yue-Feng She) Department of Mathematics, Nanjing
		University, Nanjing 210093, People's Republic of China}
	\email{{\tt she.math@smail.nju.edu.cn}}

	\begin{abstract}
	Let $p\equiv1\pmod3$ be a prime . We study several topics on additive decompositions concerning the set $C_p$ of all non-zero cubes in the finite field of $p$ elements. For example, we show that
	when $p>184291$ , the set $C_p$ has no decomposition of the form $C_p=A+B+C$ with $|A|,|B|,|C|\ge2$.
	\end{abstract}
	
	\thanks{
		\newline\indent The first author was supported by the National Natural Science Foundation of China
        (Grant No. 12101321 and Grant No. 11971222) and the Natural Science Foundation of the Higher Education Institutions of Jiangsu Province (Grant No. 21KJB110002).}

	\maketitle
	\section{Introduction}	
	\setcounter{lemma}{0}
	\setcounter{theorem}{0}
	\setcounter{corollary}{0}
	\setcounter{remark}{0}
	\setcounter{equation}{0}
	\setcounter{conjecture}{0}

For each prime $p$, let $\mathbb{F}_p$ denote the finite field of $p$ elements. For any non-empty subsets $A_1,A_2,\cdots,A_k\subseteq\mathbb{F}_p$. Set
$$A_1+A_2+\cdots+A_k=\{a_1+\cdots+a_k:\ a_j\in A_j\ \text{for}\ j=1,\cdots,k\}.$$
Also, $|A|$ denotes the cardinality of a set $A$.

In 2012, S\'{a}rk\"{o}zy \cite{Sarkozy} investigated the additive decompositions concerning the set $R_p$ of all quadratic residues modulo $p$ and posed the following interesting conjecture in \cite[Conjecture 1.6]{Sarkozy}.

\begin{conjecture}[S\'{a}rk\"{o}zy]
$R_p$ has no decomposition of the form $R_p=A+B$ with $|A|,|B|\ge2$ if $p$ is sufficiently large.
\end{conjecture}
This conjecture seems beyond reach. On the other hand, under the assumption that $A+B=R_p$ with $|A|,|B|\ge2$, the bounds of $|A|$ and $|B|$ have been extensively investigated. For example, S\'{a}rk\"{o}zy \cite[Theorem 2.1]{Sarkozy} showed that if $p$ is large enough and $A+B=R_p$ with $|A|,|B|\ge2$, then
$$\frac{\sqrt{p}}{3\log{p}}<|A|,|B|<\sqrt{p}\log{p}.$$
The factor $\log{p}$ was later removed by Shkredov \cite{Shkredov}. Moreover, in the same paper Shkredov showed that $A+A\neq R_p$ for any $A$ with $|A|\ge2$. In 2021, Chen and Yan \cite[Theorem 1.1]{Chen} improved S\'{a}rk\"{o}zy's result and proved that
$$\frac{7-\sqrt{17}}{6}\sqrt{p}+1\le |A|,|B|\le \frac{7+\sqrt{17}}{4}\sqrt{p}-6.63$$
for each odd prime $p$. As a consequence of this improvement, Chen and Yan \cite[Theorem 1.2]{Chen} showed the following beautiful result:
$$A+B+C\neq R_p$$
for any $A,B,C\subseteq\mathbb{F}_p$ with $|A|,|B|,|C|\ge2$. This refines S\'{a}rk\"{o}zy's result \cite[Theorem 1.2]{Sarkozy}. 
Recently, Chen and Xi \cite{Chen and Xi} obtained
$$\frac{1}{4}\sqrt{p}+\frac{1}{8}\le |A|,|B|\le 2\sqrt{p}-1$$
for each odd prime and studied the representation function 
$$r_{A,B}(x)=|\{(a,b)\in A\times B:\ a+b=x\}|.$$

Motivated by the above results, in this paper, let $C_p$ be the set of all non-zero cubes in $\mathbb{F}_p$. We study several topics on additive decompositions of $C_p$. Note that if $p\equiv2\pmod3$, then $C_p=\mathbb{F}_p\setminus\{0\}$. Hence we just need consider the case $p\equiv1\pmod3$. 

We now state our first result.
	
\begin{theorem}\label{Thm. A}
	Let $p\equiv1\pmod3$ be a prime with $p\ge9096$. Suppose $A+B=C_p$ with $|A|,|B|\ge2$. Then
	$$\frac{\sqrt{p}}{18}\le |A|,|B|\le 3\sqrt{p}+269.$$
	Moreover, if $p$ is large enough, then we further have 
	$$(\frac{1}{9}-o(1))\sqrt{p}\le|A|,|B|.$$
\end{theorem}

As a direct consequence of Theorem \ref{Thm. A}, we obtain the following result concerning the $3$-decompositions of $C_p$.

\begin{theorem}\label{Thm. B}
	Let $p\equiv1\pmod3$ be a prime with $p>184291$. Then $C_p$ does not have a non-trivial $3$-decomposition, i.e.,
	$$A+B+C\neq C_p,$$
	for any $A,B,C\subseteq\mathbb{F}_p$ with $|A|,|B|,|C|\ge2$.
\end{theorem}

On the other hand, in the case $|A|=|B|$, we can obtain the following improvement of Theorem \ref{Thm. A}.
\begin{theorem}\label{Thm. C}
	Let $p\equiv1\pmod3$ be a prime. Suppose $A+B=C_p$ with $|A|=|B|$. Then
	$$\sqrt{\frac{p-1}{3}}\le |A| \le\sqrt{p}.$$
	Moreover, if $A+A=C_p$, then
	$$\frac{-1}{2}+\frac{\sqrt{3}}{6}\sqrt{8p-5}\le|A|\le\sqrt{p}.$$
\end{theorem}

Bachoc, Matolcsi and Ruzsa \cite{Bach}, as well as Shkredov \cite{Shkredov} independently obtained that if $A-A=R_p\cup\{0\}$, then
$$p\ge\begin{cases}
	|A|^2+|A|-1&\mbox{if}\ |A|\ \text{is even},\\
	|A|^2+2|A|-1&\mbox{otherwise}.
\end{cases} $$
Inspired by their results, our next result concerns $A-A=C_p\cup\{0\}$.

\begin{theorem}\label{Thm. D}
	Let $p\equiv1\pmod3$ be a prime. Suppose $A-A=C_p\cup\{0\}$. Then
	$$\sqrt{\frac{p+2}{3}}\le|A|
	\le\begin{cases}
		\sqrt{p}&\mbox{if}\ |A|\equiv 0\pmod 3,\\
		\frac{1+\sqrt{17}}{4}\sqrt{p}+\frac{1+\sqrt{17}}{4}&\mbox{otherwise.}
	\end{cases} $$
\end{theorem}
\begin{remark}
	Note that $\frac{1+\sqrt{17}}{4}=1.28\cdots$.
\end{remark}

The outline of this paper is as follows. In section 2, we will introduce some necessary lemmas for the proof of Theorem \ref{Thm. A}. The proofs of Theorems \ref{Thm. A}--\ref{Thm. B} will be given in section 3. In section 4, we will give some preparations for the proofs of Theorems \ref{Thm. C}--\ref{Thm. D}. The proofs of Theorems \ref{Thm. C}--\ref{Thm. D} will be given in section 5.

\section{Preparations for the Proof of Theorem \ref{Thm. A}}	
\setcounter{lemma}{0}
\setcounter{theorem}{0}
\setcounter{corollary}{0}
\setcounter{remark}{0}
\setcounter{equation}{0}
\setcounter{conjecture}{0}

For each prime $p$, let $\mathbb{F}_p^{\times}=\{x\in\mathbb{F}_p:\ x\neq0\}$ and let $\widehat{\mathbb{F}_p^{\times}}$ be the set of all multiplicative characters of $\mathbb{F}_p$. In addition, let $\chi\in\widehat{\mathbb{F}_p^{\times}}$ be a character of order $3$ and define
$$\psi(x)=\chi(x)+\chi(x^2)$$
for any $x\in\mathbb{F}_p$. Then
$$\psi(x)=\begin{cases}
	2&\mbox{if}\ x\in C_p,\\
	0&\mbox{if}\ x=0,\\
	-1&\mbox{otherwise}.
\end{cases} $$
For simplicity, we write $\sum_{x}$ instead of $\sum_{x\in\mathbb{F}_p}$. 

We begin with the following result which is known as the Weil Theorem (cf. \cite[Theorem 5.41]{LN}).

\begin{lemma}\label{Lem. Weil's theorem}{\rm (Weil's Theorem)}
	Let $\chi_p\in\widehat{\F_p^{\times}}$ be a character of order $m>1$ and let $f(x)\in\F_p[x]$ be a
	monic polynomial which is not of the form $g(x)^m$ for any $g(x)\in\F_p[x]$. Then for any $a\in\F_p$ we have
	\begin{equation*}
		\left|\sum_{x}\chi_p(af(x))\right|\le (r-1)p^{1/2},
	\end{equation*}
 where $r$ is the number of distinct roots of $f(x)$ in $\overline{\F_p}$.
\end{lemma}

Note that $C_{13}=\{1,5,8,12\}\subseteq\mathbb{F}_{13}$ has a decomposition $C_{13}=A+B,$
where $A=\{1,5\}$ and $B=\{0,7\}$. When $p>13$ we have the following result.

\begin{lemma}\label{Lem. k>2}
	Let $p\equiv1\pmod3$ be a prime greater than $13$. Suppose that $A+B=C_p$ with $|A|\ge|B|=k\ge2$. Then $k\ge3$.
\end{lemma}

\begin{proof}
	We prove this lemma by contradiction. Suppose $k=2$. Note that for any $y\in\mathbb{F}_p$ we have
	$$(A+y)+(B-y)=A+B.$$
	By this we may assume $0\in B$ and hence $A\subseteq C_p$. Set $B=\{0,b\}$. Given an element $x\in C_p$, if $x\in A$, then
	$$x+b\in A+B=C_p.$$
	If $x\not\in A$, since $A+B=C_p$, then
	$$x-b\in A\subseteq C_p.$$
	Hence for any $x\in C_p$ either $x+b\in C_p$ or $x-b\in C_p$. Let
	$$H(x)=(1+\psi(x))\cdot(2-\psi(x+b))\cdot(2-\psi(x-b)).$$
	Then $H(x)=0$ for any $x\in\mathbb{F}_p^{\times}$. Note also that
	$$0\le H(0)=\sum_{x}H(x)=(2-\psi(b))^2\le 9.$$
	
	On the other hand, $\sum_{x}H(x)$ is equal to
	\begin{align*}
		&\sum_{x}4+4\sum_{x}\psi(x)-2\sum_{x}\psi(x+b)
		-2\sum_{x}\psi(x-b)\\
		-2&\sum_{x}\psi(x)\psi(x+b)-2\sum_{x}\psi(x)\psi(x-b)+\sum_{x}\psi(x-b)\psi(x+b)\\
		+&\sum_{x}\psi(x)\psi(x-b)\psi(x+b).
	\end{align*}
	Via a computation, $\sum_{x}H(x)$ is equal to
	$$4p-4\sum_{x}\psi(x)\psi(x+b)+\sum_{x}\psi(x)\psi(x+2b)+\sum_{x}\psi(x)\psi(x+b)\psi(x+2b).$$
	Recall that $\psi(x)=\chi(x)+\chi(x^2)$. When $p>85$ by Lemma \ref{Lem. Weil's theorem} we obtain
	$$9\ge H(0)=\sum_{x}H(x)\ge 4p-36\sqrt{p}>9,$$
	which is a contradiction. When $p\le 85$ and $p\neq13$, by computations one can verify the desired result.
	This completes the proof.
\end{proof}

\begin{lemma}\label{Lem. another bound}
	Let $p\equiv1\pmod3$ be a prime. Suppose $A+B=C_p$ with $|A|\ge|B|=k\ge2$. Then
	$$\frac{p-1}{3k}\le |A| \le \frac{p+(3^{k-1}\cdot 2k-2^k+1)\sqrt{p}}{3^k}\le \frac{p}{3^k}+\frac{2k}{3}\sqrt{p}.$$
\end{lemma}

\begin{proof}
	As $A+B=C_p$, we have
	$$|A||B|\ge|A+B|=|C_p|=\frac{p-1}{3}$$
	and hence $|A|\ge\frac{p-1}{3k}$.
	
	On the other hand, set $B=\{b_1,b_2,\cdots,b_k\}$ and let
	$$h(x)=\frac{1}{3^k}\prod_{j=1}^k\left(1+\psi(x+b_j)\right).$$
	Note that $h(x)\ge0$ for any $x\in\mathbb{F}_p$ and $h(x)=1$ for any $x\in A$. By Lemma \ref{Lem. Weil's theorem} we obtain
	\begin{align*}
		|A|\le\sum_{x}h(x)&=\frac{1}{3^k}\sum_{x}\prod_{j=1}^k\left(1+\psi(x+b_j)\right)\\
		   &=\frac{1}{3^k}\sum_{x}\left(1+\sum_{j=1}^{k}\psi(x+b_j)
		   +\sum_{i=2}^{k}\sum_{1\le j_1<\cdots<j_i\le k}\prod_{i=1}^r\psi(x+b_{j_r})\right)\\
		   &=\frac{1}{3^k}\left(p+\sum_{i=2}^k\sum_{1\le j_1<\cdots<j_i\le k}\sum_{x}\prod_{r=1}^i\psi(x+b_{j_r})\right)\\
		   &\le\frac{1}{3^k}\left(p+\sqrt{p}\sum_{i=2}^k\binom{k}{i}2^i(i-1)\right)\\
		   &=\frac{1}{3^k}\left(p+(3^{k-1}\cdot 2k-3^k+1)\sqrt{p}\right)\le\frac{p}{3^k}+\frac{2k}{3}\sqrt{p}.
	\end{align*}

	This completes the proof.
\end{proof}

Note that this lemma implies that
$$\left(\frac{1}{3k}-\frac{1}{3^k}\right)p-\frac{2k}{3}\sqrt{p}-\frac{1}{3k}\le 0.$$
Hence we see that $A$ and $B$ are sufficiently large sets as $p\rightarrow\infty$.

We next introduce some necessary notations. Let
$$L(\mathbb{F}_p)=\{f:\ f\ \text{is a complex function on $\mathbb{F}_p$}\}.$$
For any $f\in L(\mathbb{F}_p)$ and a positive integer $k$, the function $C_{k+1}(f): \mathbb{F}_p^k\rightarrow\mathbb{C}$ is defined by
$$C_{k+1}(f)(x_1,\cdots,x_k)=\sum_{x}f(x)f(x+x_1)\cdots f(x+x_k).$$
In addition, for any $f,g\in L(\mathbb{F}_p)$ we define
$$(f\circ g)(x)=\sum_{y}f(y)g(x+y).$$
Shkredov \cite[Lemma 2.1]{Shkredov} obtained the following result.
\begin{lemma}[Shkredov]\label{Lem. Shkredov on C(f)}
For any $f,g\in L(\mathbb{F}_p)$ and a positive integer $k$,
	\begin{equation*}
		\sum_{x}(f\circ g)^{k+1}(x)
		=\sum_{x_1,\cdots,x_k\in\mathbb{F}_p}C_{k+1}(f)(x_1,\cdots,x_k)\cdot C_{k+1}(g)(x_1,\cdots,x_k).
	\end{equation*}
\end{lemma}

\section{Proofs of Theorems \ref{Thm. A}--\ref{Thm. B}}	
\setcounter{lemma}{0}
\setcounter{theorem}{0}
\setcounter{corollary}{0}
\setcounter{remark}{0}
\setcounter{equation}{0}
\setcounter{conjecture}{0}

For any $A\subseteq\mathbb{F}_p$, the characteristic function of $A$ is defined by
$$A(x)=\begin{cases}
	1&\mbox{if}\ x\in A\\
	0&\mbox{otherwise.}
\end{cases}$$

{\bf Proof of Theorem \ref{Thm. A}.} Suppose $A+B=C_p$ with $|A|,|B|\ge2$. Then
\begin{align*}
	\sum_{x\in B}(A\circ\psi)^4(x)
	&=\sum_{x\in B}\sum_{y_1,y_2,y_3,y_4\in\mathbb{F}_p}\prod_{j=1}^4A(y_j)\psi(y_j+x)\\
	&=\sum_{x\in B}\sum_{y_1,y_2,y_3,y_4\in A}\prod_{j=1}^42=16|A|^4|B|.
\end{align*}
Hence
\begin{equation}\label{Eq. A in the proof of Thm. A}
16|A|^4|B|=\sum_{x\in B}(A\circ\psi)^4(x)\le\sum_{x}(A\circ\psi)^4(x).
\end{equation}

On the other hand, by Lemma \ref{Lem. Shkredov on C(f)} we have
$$\sum_{x}(A\circ\psi)^4(x)=\sum_{x_1,x_2,x_3\in\mathbb{F}_p}C_4(A)(x_1,x_2,x_3)\cdot C_4(\psi)(x_1,x_2,x_3).$$
Let 
$$E=\{(x,x,0),(x,0,x),(0,x,x):\ x\in\mathbb{F}_p^{\times}\}\cup\{(0,0,0)\}.$$
By Lemma \ref{Lem. Weil's theorem} if $(x_1,x_2,x_3)\not\in E$, then
\begin{equation}\label{Eq. B in the proof of Thm. A}
	|C_4(\psi)(x_1,x_2,x_3)|\le48\sqrt{p}.
\end{equation}
If $(x_1,x_2,x_3)\in E\setminus\{(0,0,0)\}$, then there is an element $t\in\mathbb{F}_p^{\times}$ such that 
\begin{align*}
	C_{4}(\psi)(x_1,x_2,x_3)
	&=\sum_{x}\psi^2(x)\psi^2(x+t)\\
	&=\sum_{x}(2\chi_0(x)+\psi(x))(2\chi_0(x+t)+\psi(x+t))\\
	&=4(p-2)+\sum_{x}\psi(x)\psi(x+t).
\end{align*}
By this and Lemma \ref{Lem. another bound} one can verify that 
\begin{equation}\label{Eq. B1 in the proof of Thm. A}
	C_4(\psi)(x_1,x_2,x_3)\le\begin{cases}
		4p+4\sqrt{p}&\mbox{if}\ (x_1,x_2,x_3)\in E\setminus\{(0,0,0)\},\\
	6p&\mbox{if}\ (x_1,x_2,x_3)=(0,0,0).	\end{cases}
\end{equation}

By (\ref{Eq. A in the proof of Thm. A})--(\ref{Eq. B1 in the proof of Thm. A}) we obtain
\begin{align*}
	16|A|^4|B|&\le\sum_{x_1,x_2,x_3\in\mathbb{F}_p}C_4(A)(x_1,x_2,x_3)\cdot C_4(\psi)(x_1,x_2,x_3)\\
	          &\le 48\sqrt{p}|A|^4+3(4p+4\sqrt{p})|A|^2+6p|A|,
\end{align*}
i.e.,
\begin{equation}\label{Eq. C in the proof of Thm. A}
	|A|^2|B|\le 3\sqrt{p}|A|^2+\frac{3}{4}(p+\sqrt{p})+\frac{3p}{8|A|}\le 3\sqrt{p}|A|^2+\frac{3}{4}(p+\sqrt{p})+\frac{p}{8}.
\end{equation}
The last inequality follows from Lemma \ref{Lem. k>2}.

By this and $|A||B|\ge|C_p|$ we obtain
\begin{equation}\label{Eq. D in the proof of Thm. A}
	\frac{p-1}{3}|A|\le 3\sqrt{p}|A|^2+\frac{7}{8}p+\frac{3}{4}\sqrt{p}.
\end{equation}
When $p\ge9096$ one can verify that
$$0\le\frac{\frac{p-1}{3}-\sqrt{\Delta}}{6\sqrt{p}}<6,$$
where
$$\Delta=\left(\frac{p-1}{3}\right)^2-3\sqrt{p}\left(\frac{7}{2}p+3\sqrt{p}\right).$$
On the other hand, note that by Lemma \ref{Lem. another bound} we have $|A|,|B|\ge6$ if $p\ge 9096$.
Hence the quadratic inequality (\ref{Eq. D in the proof of Thm. A}) indeed implies that for $p\ge9096$ we have
\begin{equation}\label{Eq. E in the proof of Thm. A}
	|A|\ge\frac{\frac{p-1}{3}+\sqrt{\Delta}}{6\sqrt{p}}\ge\frac{\sqrt{p}}{18}.
\end{equation}
By this and (\ref{Eq. C in the proof of Thm. A}) again we further obtain that if $p\ge9096$, then
\begin{equation}\label{Eq. F in the proof of Thm. A}
	|B|\le 3\sqrt{p}+\frac{3p+3\sqrt{p}}{4|A|^2}+\frac{3p}{8|A|^3}\le 3\sqrt{p}+269.
\end{equation}
With the method as above, we can also obtain that $|B|\ge \frac{\sqrt{p}}{18}$
and $|A|\le 3\sqrt{p}+269$ if $p\ge9096$.

Moreover, if $p$ is large enough, then by (\ref{Eq. E in the proof of Thm. A}) and the above method one can verify that
\begin{equation}\label{Eq. F in the proof of Thm. A}
	|A|,|B|\ge (\frac{1}{9}-o(1))\sqrt{p}.
\end{equation}

In view of the above, we have completed the proof.\qed

To prove our next result, we need the following result which is a special case of \cite[Theorem 1.2]{GMR}.
\begin{lemma}[Gyarmati, Matolcsi and Ruzsa]\label{Lem. GMR}
	Let $p$ be an odd prime and let $A_1,A_2,\cdots, A_k\subseteq\mathbb{F}_p$ be non-empty subsets. Then
	$$\left|\sum_{j=1}^kA_j\right|^{k-1}\le\prod_{j=1}^k\left|\sum_{i\neq j}A_i\right|.$$
\end{lemma}

{\bf Proof of Theorem \ref{Thm. B}.} We prove this theorem by contradiction. Now suppose $A+B+C=C_p$ with $|A|,|B|,|C|\ge2$. Then by Lemma \ref{Lem. GMR} and Theorem \ref{Thm. A} we have
$$\left(\frac{p-1}{3}\right)^2=|A+B+C|^2\le|A+B|\cdot|B+C|\cdot|C+A|\le\left(3\sqrt{p}+379\right)^3.$$
Howeover, when $p>184291$ we have
$$\left(\frac{p-1}{3}\right)^2>\left(3\sqrt{p}+269\right)^3,$$
which deduces a contradiction.

In view of the above, we have completed the proof.\qed

\section{Preparations for the Proofs of Theorem \ref{Thm. C}--\ref{Thm. D}}	
\setcounter{lemma}{0}
\setcounter{theorem}{0}
\setcounter{corollary}{0}
\setcounter{remark}{0}
\setcounter{equation}{0}
\setcounter{conjecture}{0}

We first introduce some necessary notations. For any $f,g\in L(\mathbb{F}_p)$, the inner product $\langle f,g\rangle$ of $f,g$ is defined by
$$\langle f,g\rangle=\sum_{x}f(x)\overline{g(x)}.$$
Set
$$||f||_2=(\langle f,f\rangle)^{1/2}=\left(\sum_{x}|f(x)|^2\right)^{1/2},$$
and
$$\langle f\rangle=\langle f,1\rangle=\sum_{x}f(x).$$
Also, the convolution $f*g$ is defined by
$$(f*g)(x)=\sum_{y}f(y)g(x-y).$$
Note that
\begin{equation}\label{Eq. relations between circ and convolution}
(f*g)(x)=(f^c\circ g)(x),
\end{equation}
where $f^c(x)=f(-x)$.

\begin{lemma}\label{Lem. computation of inner product}
	Let $p\equiv1\pmod3$ be a prime and let $f,g\in L(\mathbb{F}_p)$. Then  $\langle f\circ\psi, g\circ\psi\rangle$ is equal to
	\begin{align*}
	 2p\langle f,g\rangle-2\langle f\rangle\cdot\langle \bar{g}\rangle+J(\chi,\chi)\langle\chi^2,\bar{f}\circ g\rangle+J(\chi^2,\chi^2)\langle\chi,\bar{f}\circ g\rangle,
	\end{align*}
where $J(\chi,\chi)$ is the Jacobi sum.
\end{lemma}
\begin{proof}
	By definition
\begin{align*}
		\langle f\circ\psi,  g\circ\psi\rangle
		&=\sum_{x}(f\circ\psi)(x)\cdot\overline{( g\circ\psi)(x)}\\
		&=\sum_{y,z}f(y)\overline{g(z)}\sum_{x}\psi(x+y)\psi(x+z).		
\end{align*}
If $y=z$, then
\begin{align*}
	\sum_{x}\psi(x+y)\psi(x+z)=\sum_{x}\psi(x)^2=\frac{4(p-1)}{3}+\frac{2(p-1)}{3}=2(p-1).
\end{align*}
If $y\neq z$, writing $t=y-z\neq0$, then
\begin{align*}
	\sum_{x}\psi(x+y)\psi(x+z)
	&=\sum_{x}\psi(x)\psi(x+t)\\
	&=\chi^2(t)J(\chi,\chi)+\chi(t)\overline{J(\chi,\chi)}+2J(\chi,\chi^2)\\
	&=\chi^2(t)J(\chi,\chi)+\chi(t)\overline{J(\chi,\chi)}-2.
\end{align*}
The last equality follows from $J(\chi,\chi^2)=-1$.
By above results, one can verify that $\langle f\circ\psi, g\circ\psi\rangle$ is equal to
\begin{align*}
	2p\langle f,g\rangle-2\langle f\rangle\cdot\langle \bar{g}\rangle+J(\chi,\chi)\langle\chi^2,\bar{f}\circ g\rangle+J(\chi^2,\chi^2)\langle\chi,\bar{f}\circ g\rangle.
\end{align*}
This completes the proof.
\end{proof}

\section{Proofs of Theorems \ref{Thm. C}--\ref{Thm. D}}	
\setcounter{lemma}{0}
\setcounter{theorem}{0}
\setcounter{corollary}{0}
\setcounter{remark}{0}
\setcounter{equation}{0}
\setcounter{conjecture}{0}

{\bf Proof of Theorem \ref{Thm. C}.} Suppose $A+B=C_p$. Set $|A|=a,|B|=b$. Define a function $r\in L(\mathbb{F}_p)$ by
$$(A\circ\psi)(x)=2aB(x)+r(x)$$
for any $x\in\mathbb{F}_p$. For any $x\in B$, we observe that
$$(A\circ\psi)(x)=\sum_{y}A(y)\psi(y+x)=\sum_{y\in A}\psi(x+y)=2a=2aB(x),$$
i.e., $r(x)=0$ for any $x\in B$. Hence
$$||r||_2^2=\sum_{x\not\in B}|r(x)|^2=\sum_{x\not\in B}|(A\circ\psi)(x)|^2=||A\circ\psi||_2^2-\sum_{x\in B}|(A\circ\psi)(x)|^2.$$
By Lemma \ref{Lem. computation of inner product}
\begin{align*}
	||A\circ\psi||_2^2
	&=2p\langle A,A\rangle-2\langle A\rangle^2+2\Re\left(J(\chi,\chi)\langle\chi^2,A\circ A\rangle\right)\\
	&=2pa-2a^2+2\Re\left(J(\chi,\chi)\langle\chi^2,A\circ A\rangle\right),
\end{align*}
where $\Re(z)$ is the real part of a complex number $z$. On the other hand,
$$\left|\Re\left(J(\chi,\chi)\langle\chi^2,A\circ A\rangle\right)\right|\le |J(\chi,\chi)|\cdot|\langle\chi^2,A\circ A\rangle|\le a^2\sqrt{p}.$$
The last inequality follows from $|J(\chi,\chi)|=\sqrt{p}$ and
\begin{align*}
|\langle\chi^2,A\circ A\rangle|
=\left|\sum_{x}\chi^2(x)(A\circ A)(x)\right|
&=\left|\sum_{x}\chi^2(x)\sum_{y}A(y)A(x+y)\right|\\
&=\left|\sum_{x,y\in A}\chi^2(x-y)\right|\le a^2.
\end{align*}
Note also that
\begin{align*}
\sum_{x\in B}|(A\circ\psi)(x)|^2
&=\sum_{x\in B}\sum_{y}A(y)\psi(x+y)\sum_{z}A(z)\psi(x+z)\\
&=\sum_{x\in B}\sum_{y,z\in A}\psi(x+y)\psi(x+z)=4a^2b.
\end{align*}
By the above results, we have
\begin{equation}\label{Eq. a in the proof of Thm. C}
	||r||^2_2=||A\circ\psi||_2^2-\sum_{x\in B}|(A\circ\psi)(x)|^2\le 2ap+2a^2\sqrt{p}-2a^2-4a^2b.
\end{equation}

On the other hand,
\begin{align*}
	\langle r\rangle=\sum_{x}r(x)=\sum_{x}(A\circ\psi)(x)-2a\sum_{x}B(x)=-2ab.
\end{align*}
The last equality follows from
$$\sum_{x}(A\circ\psi)(x)=\sum_{x}\sum_{y}A(y)\psi(x+y)=\sum_{y}A(y)\sum_{x}\psi(x+y)=0.$$
By the Cauchy-Schwarz inequality
$$|\langle r\rangle|^2=\left|\sum_{x\not\in B}r(x)\right|^2\le\sum_{x\not\in B}1^2\sum_{x\not\in B}r(x)^2=(p-b)\cdot||r||_2^2.$$
Combining this with (\ref{Eq. a in the proof of Thm. C}), we obtain
\begin{equation}\label{Eq. b in the proof of Thm. C}
	a+b+2ab\le p+a\sqrt{p}+\left(\frac{1}{p}-\frac{1}{\sqrt{p}}\right)ab\le p+a\sqrt{p}.
\end{equation}

Now suppose $|A|=|B|=a$. By (\ref{Eq. b in the proof of Thm. C}) we obtain
$$2a^2+2a\le p+a\sqrt{p}.$$
This implies that
$$a\le\frac{\sqrt{p}-2+\sqrt{(\sqrt{p}-2)^2+8p}}{4}<\sqrt{p}.$$
Moreover, if $A=B$, then
$$\binom{a}{2}+a\ge |C_p|.$$
This implies
$$a\ge\frac{-1}{2}+\frac{\sqrt{3}}{6}\sqrt{8p-5}.$$

In view of the above, we have completed the proof.\qed

To prove our last result, we need the following result due to Shkredov \cite[Lemma 3.1]{Shkredov}.
\begin{lemma}[Shkredov]\label{Lem. Shkredov trick}
	Let $c$ be an integer and let $f:\mathbb{F}_p\rightarrow\mathbb{Z}$ be a function. Then
	\begin{equation}\label{Eq. a1 of Shkredov}
		||f||_2^2\ge c|\langle f\rangle|-(c-1)\cdot\left|\sum_{x: 0<|f(x)|<c}f(x)\right|.
	\end{equation}
Further,
\begin{equation}\label{Eq. a2 of Shkredov}
	||f||_2^2=c\langle f\rangle+\sum_{k\in\mathbb{Z}}N_k(k^2-ck),
\end{equation}
where $N_k=|\{x\in\mathbb{F}_p:\ f(x)=k\}|$.
\end{lemma}

Now we prove our last result.

{\bf Proof of Theorem \ref{Thm. D}.} Set $|A|=a$. The case $a=1$ is trivial. We now suppose $a\ge2$ and
$A-A=C_p\cup\{0\}$. Note that
for any $y\in\mathbb{F}_p$ we have
$$(A-y)-(A-y)=A-A.$$
By this we may assume $0\in A$ and hence $\emptyset\ne A\setminus\{0\}\subseteq C_p$. Let $s\in L(\mathbb{F}_p)$ defined by
$$(A*\psi)(x)=2(a-1)A(x)+s(x)$$
for any $x\in\mathbb{F}_p$. Then for any $x\in A$ we have
$$(A*\psi)(x)=\sum_{y}A(y)\psi(x-y)=\sum_{y\in A}\psi(x-y)=2(a-1)=2(a-1)A(x),$$
i.e., $s(x)=0$ for any $x\in A$.  By (\ref{Eq. relations between circ and convolution}) we have
$A*\psi=A^c\circ\psi$. Hence by the method appeared in the proof of Theorem \ref{Thm. C} we obtain
\begin{equation}\label{Eq. 2 norm of s}
	||s||_2^2\le 2ap+2a^2\sqrt{p}-2a^2-4(a-1)^2a
\end{equation}
and
\begin{equation}\label{Eq. langle s rangle}
\langle s\rangle=-2a(a-1).
\end{equation}
Moreover, as $||s||_2\ge 0$, we obtain a trivial upper bound of $a$, i.e.,
\begin{equation}\label{Eq. lower bound of p}
	\sqrt{p}\ge \frac{-a+\sqrt{9a^2-12a+8}}{2}= \frac{-a+\sqrt{5a^2+4a^2-12a+8}}{2}\ge \frac{\sqrt{5}-1}{2}a.
\end{equation}

On the other hand, for any $x\not\in A$ we have
\begin{align*}
	s(x)=\sum_{y}A(y)\psi(x-y)
	&=\sum_{y\in A}\psi(x-y)\\
	&=2|C_p\cap(x-A)|-|D_p\cap(x-A)|\\
	&=3|C_p\cap(x-A)|-a,
\end{align*}
where $D_p=\{x\in\mathbb{F}_p:\ \psi(x)=-1\}$. The above result implies that
$$s(x)\equiv -a\pmod3$$
for any $x\not\in A$.

{\bf Case 1.} $a\equiv0\pmod3$.

Applying $c=3$ to (\ref{Eq. a1 of Shkredov}), we obtain $$||s||_2^2\ge 3|\langle s\rangle|.$$
By (\ref{Eq. 2 norm of s}) and (\ref{Eq. langle s rangle}) we have
$$2a^2-\sqrt{p}a-p-1\le 0$$
and hence
$$a\le \frac{\sqrt{p}+\sqrt{9p-8}}{4}<\sqrt{p}.$$

{\bf Case 2.} $a\not\equiv0\pmod3$.

In this case, as $s(x)\equiv-a\pmod3$ for any $x\not\in A$, we have
$$s(x)=0\Leftrightarrow x\in A.$$
Hence $N_0=a$. On the other hand, note that
$$s(x)=-1\Leftrightarrow\sum_{y\in A}\psi(x-y)=-1\Leftrightarrow G(x)=-1,$$
where 
$$G(x)=\sum_{y\in A}\left((x-y)^{\frac{p-1}{3}}+(x-y)^{\frac{2(p-1)}{3}}\right).$$
Hence
$$N_{-1}\le \frac{2(p-1)}{3}.$$
Applying $c=-2$ to (\ref{Eq. a2 of Shkredov}) we obtain
\begin{align*}
	||s||_2^2
	&=-2\langle s\rangle+\sum_{k\in\mathbb{Z}}N_k(k^2+2k)\\
	&=4a(a-1)-N_{-1}+\sum_{k\neq -1}N_k(k^2+2k)\\
	&\ge 4a(a-1)-N_{-1}+3\sum_{k\neq -1,0}N_k\\
	&=4a(a-1)-4N_{-1}-3N_0+3p\\
	&\ge 4a(a-1)-\frac{8(p-1)}{3}-3a+3p.	
\end{align*}
This gives 
$$4a^2-(2+2\sqrt{p})a-2p-3\le \frac{p+8}{3a}.$$
As $a^2\ge 1+|C_p|$, by (\ref{Eq. 2 norm of s}) we obtain
\begin{align*}
4a^2-(2+2\sqrt{p})a-2p-3\le \frac{p+8}{3a}
&\le \frac{p+8}{\sqrt{3}\sqrt{p+2}}\\
&= \frac{\sqrt{p+2}}{\sqrt{3}}+\frac{2\sqrt{3}}{\sqrt{p+2}}\\
&\le \frac{\sqrt{p+2}}{\sqrt{3}}+\frac{2}{\sqrt{3}}.
\end{align*}
By this we finally obtain
\begin{align*}
	a
	&\le \frac{2+2\sqrt{p}+\sqrt{(2+2\sqrt{p})^2+16(2p+3+\frac{\sqrt{p+2}}{\sqrt{3}}+\frac{2}{\sqrt{3}})}}{8}\\
	&\le \frac{2+2\sqrt{p}+\sqrt{(2+2\sqrt{p})^2+16(2+2\sqrt{p})^2}}{8}\\
	&\le \frac{1+\sqrt{17}}{4}\sqrt{p}+\frac{1+\sqrt{17}}{4}.
\end{align*}

In view of the above, we have completed the proof.\qed

\Ack\ We would like to thank Prof. Hao Pan for his helpful suggestions and steadfast encouragement.

\end{document}